\documentclass{IEEEcsmag}

\usepackage[colorlinks,urlcolor=blue,linkcolor=blue,citecolor=blue]{hyperref}

\usepackage{upmath}

\jvol{XX}
\jnum{XX}
\paper{8}
\jmonth{Mar/Apr}
\jname{IT Professional}
\pubyear{2023}
\newtheorem{theorem}{Theorem}
\newtheorem{lemma}{Lemma}

\usepackage{algorithm}
\usepackage[algo2e,linesnumbered,boxed,ruled,vlined]{algorithm2e}
\usepackage{amsmath}
\usepackage{amsopn}
\usepackage{amssymb}
\usepackage{array}
\usepackage{bm}
\usepackage{booktabs}
\usepackage{caption}
\usepackage{fix-cm}
\usepackage{hyperref}
\usepackage{lineno}
\usepackage[cal=boondoxo]{mathalfa}
\usepackage{mathtools}
\usepackage{microtype}
\usepackage{multirow}
\usepackage{pgf}
\usepackage{ragged2e}
\usepackage[onehalfspacing]{setspace}
\usepackage{subcaption}
\usepackage{tabularx}
\usepackage{tikz}
\usepackage{ulem}
\usepackage{url}
\usepackage{verbatim}

\usetikzlibrary{calc}
\usetikzlibrary{plotmarks}

\newtheorem{corollary}[theorem]{Corollary}

\newcommand{\fat}[1]{\ifmmode\bm{#1}\else\textbf{#1}\fi}

\newcommand{\set}[1]{\mathbb{#1}}

\newcommand{\vect}[1]{\fat{#1}}

\newcommand{\matr}[1]{#1}

\newcommand{\tens}[1]{\mathcal{#1}}

\newcommand{\func}[1]{\textsf{#1}}

\newcommand{\order}[1]{\mathcal{O}\left( #1 \right)}

\newcommand{\mg}[0]{\matr{G}}            
\newcommand{\mq}[0]{\matr{Q}}            

\newcommand{\td}[0]{\tens{D}}            
\newcommand{\tg}[0]{\tens{G}}            
\newcommand{\ty}[0]{\tens{Y}}            
\newcommand{\tp}[0]{\tens{P}}            
\newcommand{\tz}[0]{\tens{Z}}            

\DeclareMathOperator*{\argmax}{arg\,max}

\newcommand\normb[1]{\ensuremath{\bigl\lVert#1\bigr\rVert}}

\def\textfor{\text{ for }\,}
\def\indft#1_#2,#3..#4{#1_{#2},\,#1_{#3},\,\ldots,\,#1_{#4}}  
\def\indftsmall#1_#2,#3..#4{#1_{#2},\,\ldots,\,#1_{#4}}  
\def\indt#1_#2{#1_1,\,\ldots,\,#1_{#2}}  
\def\rangeo#1{\range1,2..{#1}}
\def\range#1,#2..#3{#1,\,#2,\,\ldots,\,#3}
\def\alld{\,:\,}
\def\plusl#1#2{#1_{\oplus#2}}
\def\sumnN#1_#2^#3{\ensuremath{\sum_{%
{#1}_{#2}=1%
}^{%
{#3}_{#2}%
}%
}}
\def\sumnNseq#1_#2,#3..#4^#5{\sumnN#1_#2^{#5}\sumnN#1_#3^{#5}\cdots\sumnN#1_#4^{#5}%
}
\def\sumnNseqSmall#1_#2,#3..#4^#5{\sumnN#1_#2^{#5}\cdots\sumnN#1_#4^{#5}%
}
\def\abs|#1|{\ensuremath{\left|#1\right|}}

\begin{document}

\sptitle{\hbox{}} 
\editor{\hbox{}}


\title{Optimization of Functions Given in the Tensor Train Format}

\author{A. Chertkov}
\affil{Skolkovo Institute of Science and Technology}

\author{G. Ryzhakov}
\affil{Skolkovo Institute of Science and Technology}

\author{G. Novikov}
\affil{Skolkovo Institute of Science and Technology}

\author{I. Oseledets}
\affil{Skolkovo Institute of Science and Technology, 
}


\begin{abstract} 
Tensor train (TT) format is a common approach for computationally efficient work with multidimensional arrays, vectors, matrices, and discretized functions in a wide range of applications, including computational mathematics and machine learning.
In this work, we propose a new algorithm for TT-tensor optimization,
which leads to very accurate approximations for the minimum and maximum tensor element.
The method consists in sequential tensor multiplications of the TT-cores with an intelligent selection of candidates for the optimum.
We propose the probabilistic interpretation of the method, and make estimates on its complexity and convergence.
We perform extensive numerical experiments with random tensors and various multivariable benchmark functions with the number of input dimensions up to $100$.
Our approach generates a solution close to the exact optimum for all model problems, while the running time is no more than $50$ seconds on a regular laptop.
\end{abstract}

\maketitle

\chapterinitial{Tensor train (TT) format}~\cite{oseledets2011tensor} is a powerful paradigm for multidimensional arrays (tensors).
An arbitrary tensor can be transformed into a TT-decomposition, which is a compact (low-rank) parametric representation.
The TT-decomposition can be constructed by robust existing algorithms from an explicit tensor (that is, a complete array stored in the memory of a computing device), implicit tensor (i.\,e., a tensor given as some computational procedure for calculating any its element), or even random training dataset.
TT-format has been successfully applied in a wide range of applications~\cite{cichocki2016tensor}, including compression and acceleration of deep neural networks, image and video processing, solution of differential equations, etc.
However, to date, there is no stable approach for TT-tensor optimization, and in this work we propose a new algorithm \func{optima\_tt}, which leads to very accurate approximations for the minimum and maximum value of the given TT-tensor.
To summarize, our main contributions are the following:
\begin{itemize}
    \item We develop the new method \func{optima\_tt} for optimization of the TT-tensors based on the sequential multiplications of the TT-cores with an intelligent selection of candidates for the optimum.
    \item We establish a connection of the \func{optima\_tt} with the probabilistic approach, and we prove estimates for the complexity and convergence of the method.
    \item We implement\footnote{
        The proposed approach is implemented within the software product \func{teneva}, which is available from \url{https://github.com/AndreiChertkov/teneva}.
    } the proposed algorithms
    as a part of a publicly available python package.
    \item We check\footnote{
        The program code with numerical examples, given in this work, is publicly available in the repository \url{https://github.com/AndreiChertkov/teneva_research_optima_tt}.
    } the efficiency and stability of the \func{optima\_tt} on a wide range of model problems, including multivariable benchmarks and random TT-tensors.
\end{itemize}

\section{BACKGROUND}
\label{s:tt}

\begin{figure*}[t!]
    \centering
    \includegraphics[scale=0.5]{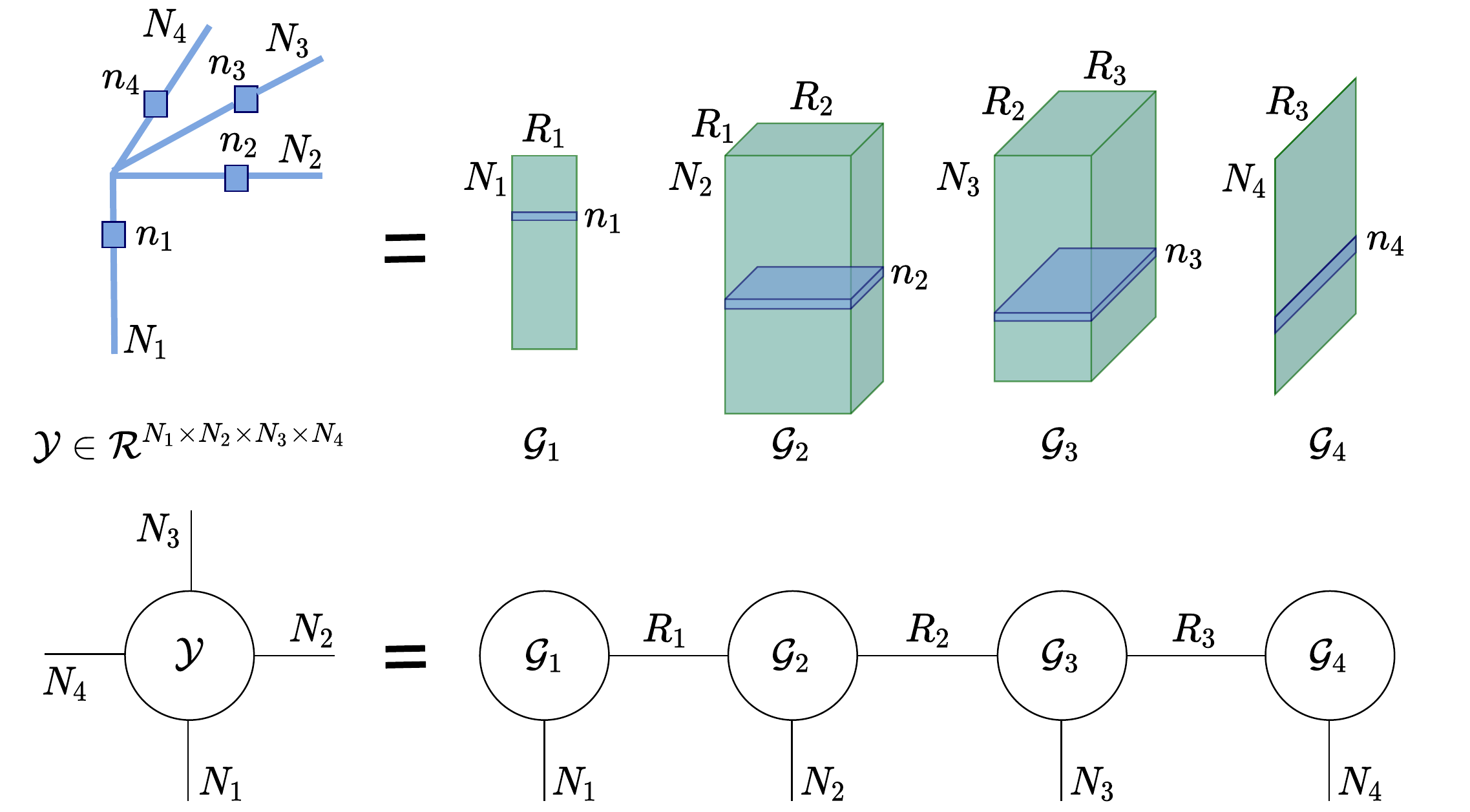}
    \vspace{0.5cm}
    \caption{
        Schematic representation of the TT-decomposition for the $4$-dimensional tensor $\ty \in \set{R}^{N_1 \times N_2 \times N_3 \times N_4}$.
        The top picture shows an illustration for the calculation of the specific tensor element $(n_1, n_2, \ldots, n_d)$ from its TT-representation according to the formula \eqref{eq:tt-repr-tns}.
        The bottom picture presents the related tensor network diagram
    }
    \label{fig:tt-element}
\end{figure*}

A tensor\footnote{
    A tensor is just a multidimensional array with a number of dimensions $d$ ($d \geq 1$).
    A two-dimensional tensor ($d = 2$) is a matrix, and when $d = 1$ it is a vector.
    For scalars we use normal font ($a, b, c, \ldots$), we denote vectors with bold letters ($\vect{a}, \vect{b}, \vect{c}, \ldots$), we use upper case letters ($\matr{A}, \matr{B}, \matr{C}, \ldots$) for matrices, and calligraphic upper case letters ($\tens{A}, \tens{B}, \tens{C}, \ldots$) for tensors with $d > 2$.
    The $(n_1, n_2, \ldots, n_d)$th entry of a $d$-dimensional tensor $\tens{A} \in \set{R}^{N_1 \times N_2 \times \ldots \times N_d}$ is denoted by $\tens{A}[n_1, n_2, \ldots, n_d]$, where $n_k = 1, 2, \ldots, N_k$ ($k = 1, 2, \ldots, d$), and $N_k$ is a size of the $k$-th mode.
    The mode-$k$ slice of such tensor is denoted by $\tens{A}[n_1, \ldots, n_{k-1}, :, n_{k+1}, \ldots, n_d]$, and it is a vector of the length $N_k$.
} $\ty \in \set{R}^{N_1 \times N_2 \times \ldots \times N_d}$ is said to be in the TT-format~\cite{oseledets2011tensor}, if its elements are represented by the following formula
\begin{multline}\label{eq:tt-repr-tns}
\ty [n_1, n_2, \ldots, n_d]
=
\sum_{r_1=1}^{R_1}
\sum_{r_2=1}^{R_2}
\cdots
\sum_{r_{d-1}=1}^{R_{d-1}}
    \\
    \tg_1 [1, n_1, r_1]
    \;
    \tg_2 [r_1, n_2, r_2]
    \;
    \ldots
    \\
    \tg_{d-1} [r_{d-2}, n_{d-1}, r_{d-1}]
    \;
    \tg_d [r_{d-1}, n_d, 1],
\end{multline}
where $(n_1, n_2, \ldots, n_d)$ is a multi-index ($n_i = 1, 2, \ldots, N_i$ for $i = 1, 2, \ldots, d$), integers $R_{0}, R_{1}, \ldots, R_{d}$ (with convention $R_{0} = R_{d} = 1$) are named TT-ranks, and three-dimensional tensors $\tg_i \in \set{R}^{R_{i-1} \times N_i \times R_i}$ ($i = 1, 2, \ldots, d$) are named TT-cores.
The TT-decomposition~\eqref{eq:tt-repr-tns} allows to represent a tensor or a discretized multivariable function in a compact and descriptive low-parameter form, which is linear in dimension~$d$ (see illustration on Figure \ref{fig:tt-element}), i.\,e., it has less than $d \cdot \max_{i=1,\ldots,d}(N_i R_i^2)$ parameters.

Many useful algorithms (e.\,g., addition, multiplication, solution of linear systems, convolution, integration, etc.) have corresponding equivalents in the TT-format.
The complexity of these algorithms turns out to be polynomial in dimension and mode size if the TT-ranks are bounded.
It makes TT-decomposition extremely popular in a wide range of applications, including computational mathematics and machine learning.

The TT-approximation for a given tensor or discretized multivariable function may be built by efficient numerical methods, e.\,g., TT-SVD~\cite{oseledets2011tensor}, TT-ALS~\cite{chertkov2022anova}, and TT-cross~\cite{oseledets2010ttcross}.
In this work, we do not discuss these methods in detail, but we assume that the TT-approximation is available, and we pose the problem of optimizing the corresponding TT-tensor.
A detailed description of the TT-format and related algorithms are given in works~\cite{oseledets2011tensor, cichocki2016tensor}.
For further presentation, we need only four operations, namely,
element-wise sum ($\func{tt\_add}(\cdot, \cdot)$) and difference ($\func{tt\_dif}(\cdot, \cdot)$),
orthogonalization ($\func{tt\_orth}(\cdot)$),
and explicit construction of the constant tensor ($\func{tt\_const}((N_1, \,N_2, \,\ldots, \,N_d), \,v)$).
We formulate them in the Appendix in the form of corresponding lemmas.


\section{MAIN IDEA}
\label{s:method}

\begin{figure*}[t!]
    \centering
    \includegraphics[scale=0.48]{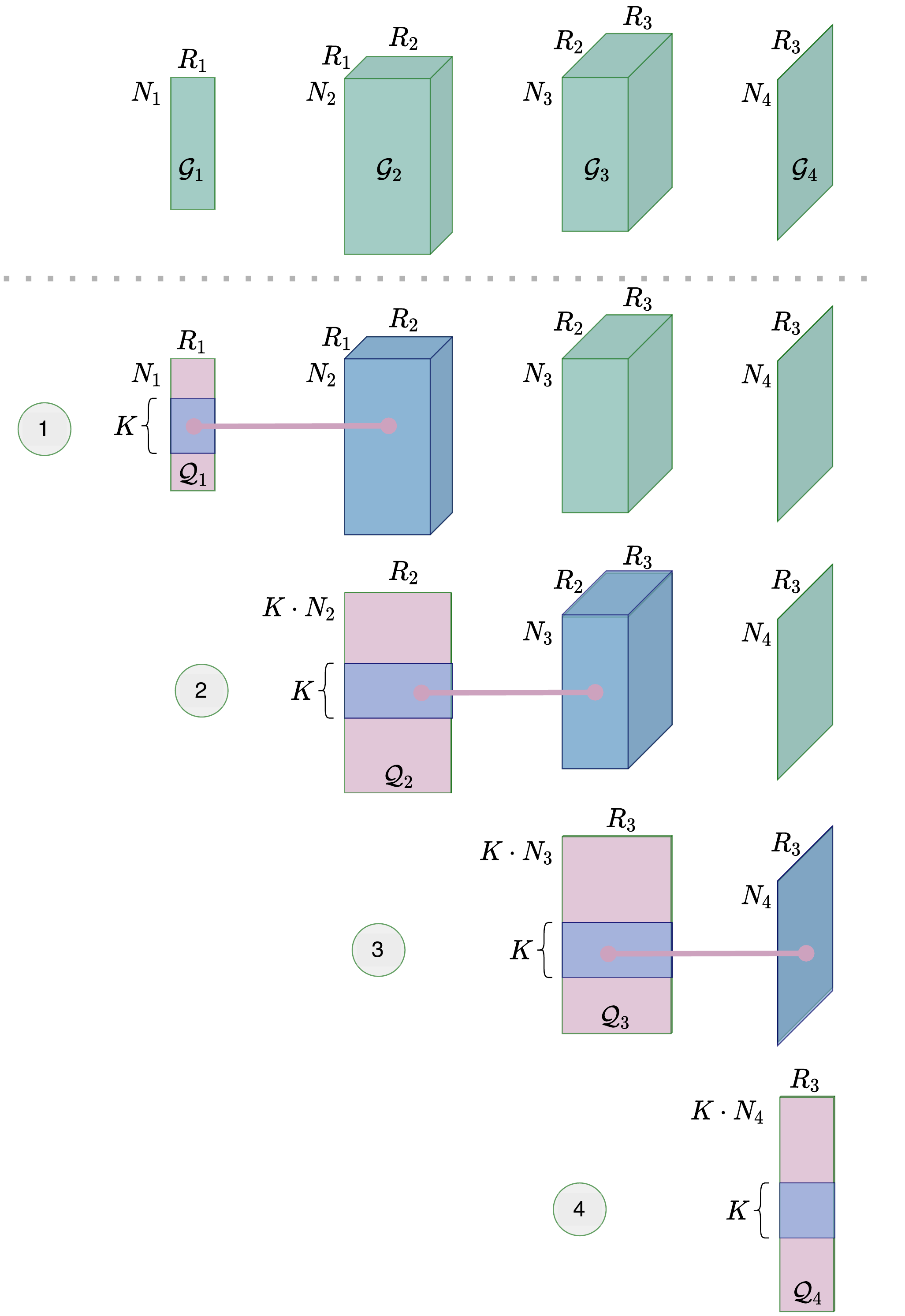}
    \vspace{0.42cm}
    \caption{
        Schematic representation of the proposed approach \func{optima\_tt\_max} for TT-tensor maximization in the $4$-dimensional case.
        For the simplicity of presentation, the rows selected at iterations are drawn as continuous blocks (they are not in practice)
    }
    \label{fig:top_k}
\end{figure*}

The problem considered in this paper
is to find the maximum and minimum elements of the tensor given in the TT-format.
In this section, we provide a motivation
and general description of the proposed method for optimizing TT-tensors as well as the algorithms computation complexities.
The basic idea behind our algorithm is that
we treat the values of 
the tensor under consideration
as values of
probability density function
of some vector random variable~$\vect{\xi} = (\indt\xi_d)$, where $\xi_i\in\{\rangeo N_i\}$.
To avoid negative values, we square the tensor element-wise: 
$
\func{p}(\indt n_d)
= 
C
\left( \ty[\indt n_d] \right)^2
$
for all $n_k = \rangeo N_k$ ($k = \rangeo d$),
where $\func{p}$ is the probability distribution of $\vect{\xi}$ and~$C$ is 
a normalization constant.
In this formulation,
the task of finding the maximum modulo element in the tensor~$\ty$
is equivalent to finding the most probable value of~$\vect{\xi}$.

The algorithm for sampling form the given TT-tensor,
which is treated as a probability distribution, is described in~\cite{dolgov2019hybrid}. 
While the algorithm from this paper is running,
the components of the vector~$\vect\xi$
from the corresponding marginal distributions
are sequentially sampled.
Our idea is to modernise this algorithm
so that at each step we take the most likely~$K$ values.
Thus, our algorithm is deterministic.
The basic idea is that
since we take the most likely components of a random vector
at each step,
we also get a vector
that has a near maximum probability
as a result.


Let us take a closer look at how this algorithm works.
First, consider
the sampling process of a random variable
whose probability density is given by the tensor $\tp$.
Here $\tp=\ty\odot\ty$,
where~$\odot$ denotes element-wise (Hadamard) product of tensors.
The sampling of
the
random variable~$\vect{\xi}$
can be done as a series of consistent samplings of univariate random variables
\begin{equation}\label{eq:one_coordinate_sample}
\xi_i \sim
    \func{p}_i(\xi_i \, | \, \indt\xi_{i - 1}),
\end{equation}
since any random variable can be represented as a product of marginal-conditional probabilities
\begin{multline}\label{eq:marginal}
\func{p}
    (\indt \xi_d) =
    \func{p}_1(\xi_1)
    \func{p}_2(\xi_2 \, | \, \xi_1)\times
    \\
    \func{p}_3(\xi_3 \, | \, \xi_1,\,\xi_2)
    \ldots
    \func{p}_d(\xi_{d} \, | \, \indt\xi_{d-1}).
\end{multline}

If in the $l$-th sampling step
we obtained elements $\{\indt\tilde\xi_{l-1}\}$,
then the marginal distribution function~$\func p_l$
is given by the following expression
\begin{multline}
    \func{p}_l(\xi_l \, | \,
    \indt\tilde\xi_{l-1}
    )
    =
    \sumnNseqSmall  n_l+1,l+2 .. d^N
    \\
    \func p(\indt\overline\xi_{l-1},\,
    \xi_l,\,
    \indft n_l+1,l+2 .. d
    ).
\end{multline}
So, if we can calculate this sums effectively, we can easy sample from the given distribution.
The idea that makes it possible for our method to work effectively is presented in the following theorem.

\begin{theorem}\label{th:orth_property}
    Let $\ty$ be a tensor with the TT-representation given by definition~\eqref{eq:tt-repr-tns}.
    Let the TT-cores of this representation are orthogonalised with use of $\func{tt\_orth}$
    so that all but the first cores meet the relationship~\eqref{eq:tt_orthogonalization}.
    Then the result of the convolution of this tensor with itself 
    by the last $l$ indices ($0<l<d$)
    is given by the following explicit expression
    \begin{multline}
    \sumnNseq n_d-l+1, d-l+2 .. d^N
    \\
    \ty[\indt n_d]\ty[\indt n_d]
    =
    \\
    \bigl\|
    \tg_1 [1, n_1, \alld]
    \;
    \tg_2 [\alld, n_2, \alld]
    \;
    \cdots
    \\
    \tg_{d-l} [\alld, n_{d-l}, \alld] 
    \bigr\|^2_2
    .
    \end{multline}
\end{theorem}
The proof of this Theorem is in Appendix.
Thus, the convolution procedure (i.\,e.\ the calculation of marginal probabilities) is reduced to several consecutive matrix-vector products without the explicit summation.
Now, with~$K$ already selected sub-indices at each~$j$-th step, we find $K \cdot N_j$ new probabilities, and from them we choose~$K$ maximum probabilities.

\subsection{Description of the main algorithm}

First, consider the problem of finding the maximum modulo element in the given TT-tensor $\ty \in \set{R}^{N_1 \times N_2 \times \ldots \times N_d}$.
We describe the corresponding method \func{optima\_tt\_max} in Algorithm~\ref{alg:optima_tt_max}, and provide a schematic graphical illustration in Figure~\ref{fig:top_k}.
In the algorithm, we use the function \func{reshape}, which changes the dimension and mode sizes of the passed multidimensional array according to the provided new array shape.
We use the function \func{stack}, which concatenates two matrices $\matr{A}_1 \in \set{R}^{N \times M_1}$ and $\matr{A}_2 \in \set{R}^{N \times M_2}$ into one long matrix $\matr{A} \in \set{R}^{N \times (M_1+M_2)}$.
As a \func{top\_k}, we denote a function that returns a list of $K$ matrix row numbers in descending order of their norms, i.\,e., for the matrix $\matr{A} \in \set{R}^{N_1 \times N_2}$ it returns the vector $\vect{i} \in \set{N}^{K}$, where $\vect{i}[j]$ corresponds to the number of the row that has a $j$-th largest norm ($j = 1, 2, \ldots, K$).
Note that if $K \geq N_1$, then \func{top\_k} will return the numbers of all rows in descending order of norms.
    
The orthogonalization of TT-cores $\tg_i \in \set{R}^{R_{i-1} \times N_i \times R_i}$ ($i = d, d-1, \ldots, 2$) is performed before the main iterations.
As a result, all TT-cores to the right of any $i$-th TT-core ($i = 1, 2, \ldots, d-1$) are orthogonalised.
Then we successively multiply the unfoldings of TT-cores in pairs from left to right and choose $K$ ($K \geq 1$) rows from the result that have the maximum value of the norm.
We separately store the corresponding selected indexes in the set $\matr{I}$,
and when the last TT-core is processed,
the multi-index in the first row of~$\matr{I}$ will correspond to the approximated maximum modulo element of the TT-tensor $\ty$.
Note that the full set of multi-indices $\matr{I}$ corresponds to the approximated ``top-K'' maximum modulo values in the tensor.

The following theorem
provides a rigorous justification
for the correspondence between
the described Algorithm and
the mathematical formulation
of taking the~$K$ most probable components of the vector random variable~$\vect{\xi}$
as previously described (the proof is presented in Appendix).
 
\begin{figure}[ht]
\begin{center}
\begin{algorithm}[H]
\small
\SetAlgoLined
\KwData{
    $d$-dimensional TT-tensor $\ty \in \set{R}^{N_1 \times N_2 \times \ldots \times N_d}$, presented as a list of $d$ TT-cores $\{ \tg_1,\,\tg_2,\,\ldots,\,\tg_d \}$ ($\tg_i \in \set{R}^{R_{i-1} \times N_i \times R_i}$ for $i = 1, 2, \ldots, d$); number of selected candidates for each TT-core $K$.
}
\KwResult{
    multi-index $\vect{i_{|max|}} \in \set{N}^d$ for $\ty$ which relates to the maximum modulo element.
}

$
\func{tt\_orth}(\ty)
$

$
\mq = \tg_1 [1, :, :]
    \in \set{R}^{N_1 \times R_1} \label{alg:line:Q_first}
$

$
\matr{I} = [1, 2, \ldots, N_1]^T
    \in \set{N}^{N_1 \times 1}
$

$
\vect{ind} = \func{top\_k}(\mq, \; K)
    \in \set{N}^{K}
$
 
$
\mq \gets \mq [\vect{ind}, :]
    \in \set{R}^{K \times R_1}
$

$
\matr{I} \gets \matr{I} [\vect{ind}, :]
    \in \set{N}^{K \times 1}
$

\For{$i = 2$ \KwTo $d$}{
    $
    \mg_i = \func{reshape}
        (\tg_i, \; (R_{i-1}, N_i \cdot R_i))
    $

    $
    \mq \gets \mq \mg_i
        \in \set{R}^{K \times N_i \cdot R_i} \label{alg:line:Q_rest}
    $

    $
    \mq \gets \func{reshape}(\mq, \; (K \cdot N_i, R_i))
    $

    $
    \matr{I}_{old} = \matr{I} \otimes \matr{E}_{(N_i)}
        \in \set{N}^{K \cdot N_i \times (i-1)}
    $
    
    //
    $\matr{E}_{(N_i)} \in \set{R}^{N_i \times 1}$
    is a column vector of all ones
    
    $
    \matr{I}_{cur} = \matr{E}_{(K)} \otimes \matr{W}_{(N_i)}
        \in \set{N}^{K \cdot N_i \times 1}
    $
    
    //
    $\matr{W}_{(N_i)} \in \set{R}^{N_i \times 1}$
    is a column vector of $1, 2, \ldots, N_i$
    
    $
    \matr{I} \gets \func{stack} \left(
        \matr{I}_{old},
        \;
        \matr{I}_{cur}
    \right)
        \in \set{N}^{K \cdot N_i \times i}
    $

    $
    \vect{ind} = \func{top\_k}(\mq, \; K)
        \in \set{N}^{K} \label{alg:line:Q_rest_reshape}
    $
    
    $
    \mq \gets \mq [\vect{ind}, :]
        \in \set{R}^{K \times R_i}
    $

    $
    \matr{I} \gets \matr{I} [\vect{ind}, :]
        \in \set{N}^{K \times i}
    $
}

\Return{$\matr{I}[1, :]$}

\caption{Method \func{optima\_tt\_max} for TT-tensor maximization.}
\label{alg:optima_tt_max}
\end{algorithm}
\end{center}
\end{figure}

\begin{figure}[ht]
\begin{center}
\begin{algorithm}[H]
\small
\SetAlgoLined
\KwData{
    $d$-dimensional TT-tensor $\ty \in \set{R}^{N_1 \times N_2 \times \ldots \times N_d}$, presented as a list of $d$ TT-cores $\{ \tg_1,\,\tg_2,\,\ldots,\,\tg_d \}$ ($\tg_i \in \set{R}^{R_{i-1} \times N_i \times R_i}$ for $i = 1, 2, \ldots, d$); number of selected candidates for each TT-core $K$.
}
\KwResult{
    multi-indices $\vect i_{min} \in \set{N}^d$ and $\vect i_{max} \in \set{N}^d$ 
    which relate to the minimum and maximum elements of $\ty$, correspondingly.
}

$
\vect i_{|max|} = 
    \func{optima\_tt\_max} (\ty, \; K)
$

$
y_{|max|} = \ty [\vect i_{|max|}]
$
// Element of the TT-tensor; see~\eqref{eq:tt-repr-tns}

$
\td = \func{tt\_const}(
    \func{shape}(\ty), y_{|max|})
$

$
\tz = \func{tt\_dif}(\ty, \td)
$

$
\vect i_{|min|} =
    \func{optima\_tt\_max} (\tz, \; K)
$

$
y_{|min|} = \ty [\vect i_{|min|}]
$
// Element of the TT-tensor; see~\eqref{eq:tt-repr-tns}

\If{$y_{|max|} \geq y_{|min|}$}{
    \Return{$(\vect i_{|min|}, \; \vect i_{|max|})$}
} \Else {
    \Return{$(\vect i_{|max|}, \; \vect i_{|min|})$}
}

\caption{Method \func{optima\_tt} for TT-tensor optimization.}
\label{alg:optima_tt}
\end{algorithm}
\end{center}
\end{figure}

\begin{theorem}\label{state:probabilistic_interpretation}
For the given tensor~$\ty$ in TT-format,
Algorithm~\ref{alg:optima_tt_max} represents
the implementation of the proposed approach
based on the probabilistic interpretation $\tp=\ty\odot\ty$,
which keeps most likely~$K$ indices of the tensor $\tp$ on each step.
\end{theorem}

\subsection{Finding both minimum and maximum}

A simple approach for finding both minimum and maximum values of the given TT-tensor can be formulated based on the \func{optima\_tt\_max}.
The proposed optimization method \func{optima\_tt} is presented in Algorithm~\ref{alg:optima_tt}.
First, we find the item with the maximum modulo value $(i_{|max|}, y_{|max|})$ by \func{optima\_tt\_max} method.
Depending on the structure of the tensor, this may be the maximum or minimum element.
Then we subtract the constant TT-tensor $\td$ equal to $y_{|max|}$ from the original tensor $\ty$, i.\,e.,
we calculate the tensor $\tz = \ty \ominus \td$ in the TT-format,
where $\ominus$ denotes element-wise subtraction.
Note that the second extreme value of the tensor $\ty$ is the maximum modulo element of the tensor $\tz$.
Therefore, we can apply the \func{optima\_tt\_max} method to the tensor $\tz$ and obtain the item $(i_{|min|}, y_{|min|})$.
Finally, we determine which of the two obtained items is the minimum and which is the maximum by the comparison of values $y_{|min|}$ and $y_{|max|}$.

\section{COMPLEXITY OF THE ALGORITHM}

We can easily obtain the computational complexity of the proposed optimization approaches \func{optima\_tt\_max} and \func{optima\_tt}.
The corresponding estimate is given in the following theorem.

\begin{theorem} 
The computational complexity of the \func{optima\_tt\_max} method from Algorithm~\ref{alg:optima_tt_max} and of the \func{optima\_tt} method from Algorithm~\ref{alg:optima_tt} is
\begin{equation}\label{eq:optima_tt_complexity}
\order{d \cdot K \cdot N \cdot R^2},
\end{equation}
where $d$ is a tensor dimension, $K$ is a number of selected candidates for each TT-core, $N$ and $R$ are the typical mode size and TT-rank, respectively.
\end{theorem}
\begin{proof}
According to Algorithm~\ref{alg:optima_tt_max}, we $(d-1)$ times multiply the matrices of the size $K \times R$ and $R \times (N \cdot R)$, and the related complexity estimate is exactly~\eqref{eq:optima_tt_complexity}.
We also $d$ times calculate the row norms for the $(K \cdot N) \times R$ matrices, and this operation has complexity estimate $\order{d \cdot K \cdot R}$, which is negligible compared to~\eqref{eq:optima_tt_complexity}.
All operations described in Algorithm~\ref{alg:optima_tt} also have low complexity compared to~\eqref{eq:optima_tt_complexity}.
\end{proof}

\section{ACCURACY OF THE ALGORITHM}

The following theorem and its corollaries show what kind of error we get with our method in the worst case.

\begin{theorem}\label{trm:bad_case}
Let $\widetilde{\vect{i}} = [ \indt \widetilde i_d ]$ be the multi-index that was found by approximate search by Algorithm~\ref{alg:optima_tt_max} with~$K = 1$, let $\tp = \ty\odot\ty$ and $\vect{i}^* =[ \indt i^*_d ] = \argmax_{\vect{i}} \tp[\vect i]$.
Then
\begin{equation}\label{eq:bad_case}
\tp[\indt \widetilde i_d] \geq
    \prod_{i=2}^d\frac1{N_i}
        \tp[\indt i^*_d].
\end{equation}
\end{theorem}
\begin{proof}
Denote the tensor obtained from the given $d$-dimensional tensor $\ty$ by summing over the last $l$ indices ($1 \leq l \leq d$)  by $\plusl\ty l$, i.\,e.,
\begin{equation}
\begin{split}
\plusl \ty l
&
[\indt n_{d-l}] =
    \\
    &
    \sumnNseqSmall n_d-l+1, d-l+2 .. d^N
    \ty[\indt n_d].
\end{split}
\end{equation}
    
In the first iteration, Algorithm~\ref{alg:optima_tt_max} chooses index~$\widetilde{i}_1$, which delivers the maximum of the tensor $\plusl \tp{d-1}$, therefore
\begin{equation}\label{eq:theor_bad:fst}
\plusl \tp{d-1}[\widetilde{i}_1] \geq
    \plusl \tp{d-1}[i^*_1].
\end{equation}
Similarly, on the $j$-th step, Algorithm~\ref{alg:optima_tt_max} chooses index $\widetilde{i}_j$, which delivers the maximum of the tensor $\plusl \tp{d-j}[\indt\widetilde i_{d-1},\,i]$,
considered as a (discrete) function of the last argument~$i$, thus
\begin{multline}\label{eq:ineq_in_main_theorem_proof}
\plusl \tp{d-j}[\indt\widetilde i_{d-j}]
    \geq
    \\
    \geq
    \frac1{N_j}
    \sum_{i=1}^{N_j}
    \plusl \tp{d-j}[\indt\widetilde i_{d-j},\,i]
    =
    \\
    =
    \frac1{N_j}\plusl \tp{d-j+1}[\indt\widetilde i_{d-j}].
\end{multline}
Considering the chain of inequalities derived from the last relation by varying~$j$ from~$1$ to~$d$, and using inequality~\eqref{eq:theor_bad:fst}, we obtain
\begin{multline}
\tp[\indt \widetilde i_d]
    \geq
    \prod_{i=2}^d\frac1{N_i}
    \plusl \tp{d-1}[\widetilde i_1]
    \\
    \geq
    \prod_{i=2}^d\frac1{N_i}
    \plusl \tp{d-1}[i^*_1].
\end{multline}
Finally, note that since all elements of the tensor $\tp{}$ are non-negative, and $\tp[\indt i^*_d]$ is a summand of $\plusl \tp{d-1}[i^*_1]$, we have
\begin{equation}
\plusl \tp{d-1}[i^*_1] \geq
\tp[\indt i^*_d],
\end{equation}
which finishes the proof.
\end{proof}

Note the peculiarity of the estimate given in this Theorem: it does not depend on the size~$N_1$ of the first index.
Thus, if we combine~$j$ first indices into one which is varying from~$1$ to $N_1 N_2 \ldots N_j$, we can noticeably improve this estimate.
Namely, the following corollary holds true.

\begin{corollary}\label{col:join}
In the notation of the Theorem~\ref{trm:bad_case}, let us construct a tensor $\ty_j$ from the tensor $\ty{}$ in which first~$j$ indices are combined:
$
\ty_j[\overline{i_1 \ldots i_j}, \, \indftsmall i_{j+1},{j+2}..d] =
    \ty[\indt i_d].
$
The line above the indices denotes
the sequence number of this multi-index
in the sequence of multi-indexes
ordered lexicographically
(in little-endian convention):
\begin{multline}
\overline{i_1i_2\cdots i_j} =
    i_1
    +(i_2-1)N_1
    +(i_3-1)N_1N_2+\cdots
    \\
    +
    (i_j-1)N_1N_2\cdots N_{j-1}.
\end{multline}
Let Algorithm~\ref{alg:optima_tt_max} been applied to the tensor~$\ty_j$ results in the indices~$\{\overline{\tilde i_1\tilde i_2\cdots\tilde i_j}, \, \indft \tilde i_{j+1},{j+2}..d\}$.
Then the following estimate is true
\begin{equation}\label{eq:collary_index_join}
\tp[\indt \widetilde i_d] \geq
    \prod_{i=j+1}^d\frac1{N_i}
    \tp[\indt i^*_d].
\end{equation}
\end{corollary}
This method 
gives an increase in accuracy
of~$N_2\cdots N_j$ times in the worst-case scenario.

Note that once we have a TT-representation of a tensor, it is easy to construct a TT-representation of the tensor obtained from this one by combining the first~$j$ indices.
The elements of the first core~$\widetilde\tg_1$
of such representation
are written as
\begin{multline}
    \widetilde\tg_1(1, \,
   \overline{i_1i_2\cdots i_j}
   ,\,
    \alpha
    )
=\\
\sumnNseqSmall\alpha_2,3 .. j^r
\tg_1(1,\, i_1, \,\alpha_1)\times\\
\tg_2(\alpha_1,\, i_2, \,\alpha_2)
\cdots
\tg_j(\alpha_{j-1},\, i_j, \,\alpha).
\end{multline}
However, this representation may require high memory consumption.

Now consider the accuracy estimates of Algorithm~\ref{alg:optima_tt_max} for the case where~$K>1$.
Let $K \geq N_1 N_2 \cdots N_j$ for some~$j$, $1<j< d$.
Then on the iteration with number $j+1$, Algorithm~\ref{alg:optima_tt_max} seeks for the maximum~$K$ elements among all combination of the indices $\{\indt i_{j}\}$.
This situation is equivalent to the case where these indices
are combined into a single index.
We have estimation~\eqref{eq:collary_index_join} for such a case.
Thus the following corollary is true.

\begin{corollary}
In the notation of Theorem~\ref{trm:bad_case}, let~$K>1$ and let~$j$ be the index such that
\begin{equation}
\prod_{i=1}^{j-1}{N_i} \leq K < \prod_{i=1}^{j}{N_i}.
\end{equation}
Then the inequality~\eqref{eq:collary_index_join} holds.
\end{corollary}
\begin{proof}
To complete the proof, we need to verify the inequality similar to~\eqref{eq:ineq_in_main_theorem_proof} for the case $K>1$.
One can easily done this by a slight modification of the reasoning given in the proof of Theorem~\ref{trm:bad_case}.
\end{proof}

Note that in the case where~$K < N_1$, the worst-case algorithm gives the same result as it would if it was run at~$K=1$.
However,
in real numerical experiments with random tensors,
we have observed both an improvement in the result (including the case $K < N_1$) and, in rare cases, a deterioration compared to the case~$K=1$.

\section{NUMERICAL EXPERIMENTS}
\label{s:calc}

{
To check the accuracy and demonstrate the capabilities of the proposed optimization method \func{optima\_tt}, we carried out three series of numerical experiments.
First, we consider the optimization task for various random tensors in the TT-format.
Then we apply our approach for $10$ analytical benchmark functions that are widely used for the evaluation of optimization algorithms.
In both cases, we consider a relatively small dimension ($d \leq 6$) and mode size ($n \leq 20$), which makes it possible to estimate the accuracy of \func{optima\_tt} by comparing the result with the exact value of the optimum obtained by the simple brute-force method.

Then we consider a complex optimization problem for $5$ different $100$-dimensional benchmark functions on a fine grid ($N = 2^{10}$), for which it is possible to explicitly construct the TT-cores.
To check the accuracy of the \func{optima\_tt} in this case, we use the known location of the global minimum for all benchmarks.

We obtained high accuracy of the result in our experiments\footnote{
    In all numerical experiments,
    we choose $K=100$ (number of selected row numbers in descending order of their norm for all unfolding matrices)
    for the reliability and stability of the algorithm.
    The dependence of the accuracy of the result
    on the value of the~$K$
    is illustrated in Figure~\ref{fig:k_dep}.
    
} (in several cases the optimum was found exactly), while the computation time was no more than $40$ seconds for $100$-dimensional functions.
All calculations were carried out on a regular laptop.
}

\subsection{Random TT-tensors of small dimensions}
\label{s:calc_random_small}

\begin{table*}[t!]
\caption{
    The accuracy of the proposed method for random TT-tensors
}
\begin{center}
\begin{tabular}{|p{4.0cm}|p{3.2cm}|p{3.2cm}|p{3.2cm}|}
\hline
\multicolumn{1}{|c|}{Dimension}           &
\multicolumn{1}{ c|}{TT-rank}             &
\multicolumn{1}{ c|}{Error for $y_{min}$} &
\multicolumn{1}{ c|}{Error for $y_{max}$} \\ \hline


\multirow{5}{*}{4}
&    1 & 0        & 0        \\
&    2 & 7.11e-15 & 2.84e-14 \\
&    3 & 1.42e-14 & 1.42e-14 \\
&    4 & 2.84e-14 & 2.84e-14 \\
&    5 & 2.84e-14 & 4.26e-14 \\
\hline

\multirow{5}{*}{5}
&    1 & 0        & 0        \\
&    2 & 8.53e-14 & 5.68e-14 \\
&    3 & 5.68e-14 & 5.68e-14 \\
&    4 & 2.27e-13 & 1.14e-13 \\
&    5 & 1.71e-13 & 1.14e-13 \\
\hline

\multirow{5}{*}{6}
&    1 & 0        & 0        \\
&    2 & 1.14e-13 & 5.68e-14 \\
&    3 & 2.27e-13 & 2.27e-13 \\
&    4 & 4.55e-13 & 6.82e-13 \\
&    5 & 4.55e-13 & 6.82e-13 \\
\hline


\end{tabular}
\end{center}
\label{tab:result_random_small}
\end{table*}

\begin{figure}[t!]
    \centering
    \includegraphics[scale=0.325]{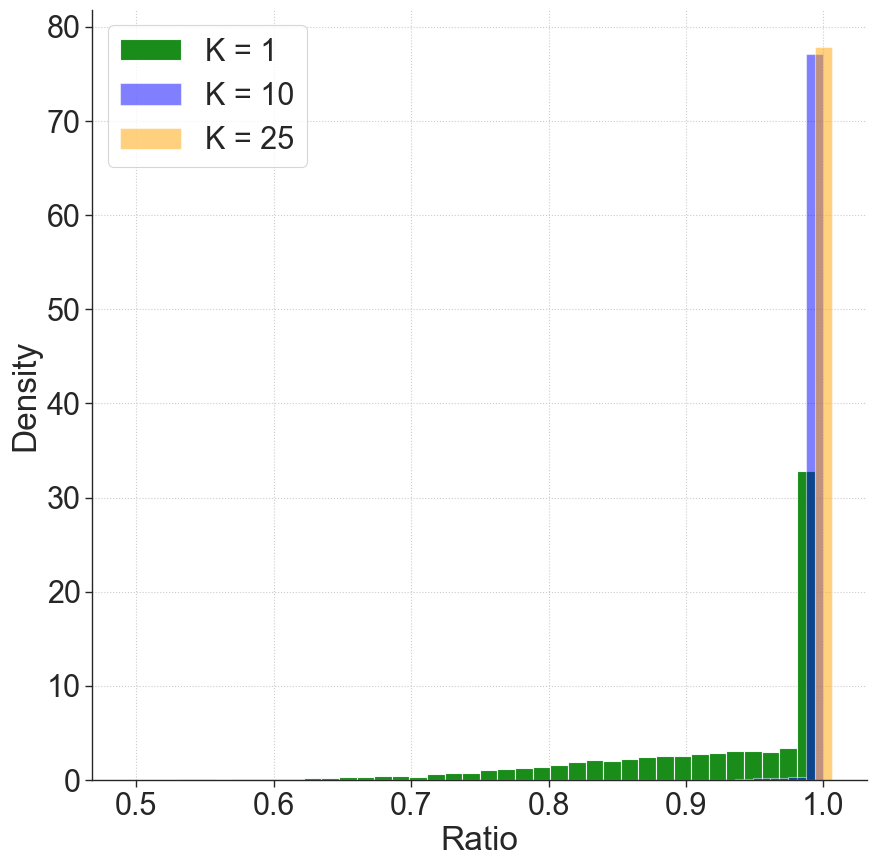}
    \caption{
        Density of the ratio of found value over true maximal element for random TT-tensor
    }
    \label{fig:k_dep}
\end{figure}

For each value of dimension~$d$ and TT-rank~$r$,
we perform the following steps $100$ times and report the maximum absolute error:
\begin{enumerate}
    \item Select different random sizes $N_1, N_2, \ldots, N_d$ ranging from $5$ to $20$ for all modes of $d$-dimensional tensor.
    \item Generate the random $d$-dimensional TT-tensor $\tens{Y} \in \set{R}^{N_1 \times N_2 \times \ldots \times N_d}$ with TT-rank equals $r$ (we use the standard Gaussian distribution to initialize the TT-cores of the random TT-tensor).
    \item Run \func{optima\_tt} algorithm for the TT-tensor $\tens{Y}$ and obtain the approximation to its minimum ($y_{min}$) and maximum ($y_{max}$) values.
    \item Transform the TT-tensor $\tens{Y}$ to the full format and find its exact minimum ($y_{min}^{(real)}$) and maximum ($y_{max}^{(real)}$) values by simple brute-force method.
    \item Calculate the absolute errors
    $
        e_{min} = |y_{min} - y_{min}^{(real)}|
    $
    and
    $
        e_{max} = |y_{max} - y_{max}^{(real)}|
    $
    to check the accuracy of the result.
\end{enumerate}
The computation results for $d = 4, 5, 6$ and $r = 1, 2, 3, 4, 5$ are presented in Table~\ref{tab:result_random_small}.
For all the cases, the maximum absolute error is not higher than $10^{-12}$, while the average time of one run was about $0.02$ sec.

Additionally, we conduct the experiment to evaluate the dependence of the result on the value of the parameter $K$.
We generate $10^4$ random $6$-dimensional TT-tensors with mode size $16$ and TT-rank $3$.
The histogram in Figure~\ref{fig:k_dep} shows the distribution of the ratio of found value over true maximal element, i.\,e., $| \frac{y_{max}^{(real)}}{y_{max}} |$.
We present the distributions for $K = 1, 10, 25$.
Even with $K=1$, we have in most cases a fairly accurate result, however, increasing the parameter value removes the corresponding rare cases of a large error.
    
\subsection{Model functions of small dimensions}
\label{s:calc_function_small}

\begin{table*}[t!]
\caption{
    The accuracy of the proposed method for $6$-dimensional analytic functions
}
\begin{center}
\begin{tabular}{|p{4.0cm}|p{3.2cm}|p{3.2cm}|p{3.2cm}|}
\hline
\multicolumn{1}{|c|}{Function}            &
\multicolumn{1}{ c|}{TT-rank}             &
\multicolumn{1}{ c|}{Error for $y_{min}$} &
\multicolumn{1}{ c|}{Error for $y_{max}$} \\ \hline


Ackley         &  11.6 & 7.11e-15 & 0        \\
\hline

Alpine         &   3.0 & 8.88e-16 & 0        \\
\hline

Dixon          &   5.4 & 2.02e-11 & 0        \\
\hline

Exponential    &   4.8 & 0        & 2.78e-17 \\
\hline

Grienwank      &  11.3 & 1.26e-13 & 0        \\
\hline

Michalewicz    &   3.5 & 1.33e-15 & 5.90e-17 \\
\hline

Qing           &   4.4 & 8.38e-06 & 1.22e-04 \\
\hline

Rastrigin      &   4.1 & 0        & 0        \\
\hline

Schaffer       &  12.0 & 1.78e-15 & 4.44e-16 \\
\hline

Schwefel       &   2.8 & 2.84e-14 & 0        \\
\hline


\end{tabular}
\end{center}
\label{tab:result_function_small}
\end{table*}

In this series of experiments, we consider $10$ popular benchmark functions~\cite{jamil2013literature} for the $6$-dimensional case.
For each benchmark $\func{f}$, we perform the following steps:
\begin{enumerate}
    \item Set the Chebyshev grid with $N_1 = N_2 = \ldots = N_d = 16$ for function discretization.
    \item Build the TT-tensor $\tens{Y} \in \set{R}^{N_1 \times N_2 \times \ldots \times N_d}$, which approximates the discretized function $\func{f}$, by TT-cross method and calculate the average TT-rank
    $r$ of $\tens{Y}$.
\end{enumerate}
Next, we carry out the same steps $3, 4, 5$ as for the case of random TT-tensors.
The computation results are presented in Table~\ref{tab:result_function_small}.
For almost all cases, the absolute error is not higher than $10^{-10}$, while the average time of one run of the optimization algorithm was about $0.2$ sec.

\subsection{Multidimensional model functions}
\label{s:calc_function_big}

\begin{table*}[t!]
\caption{
    The accuracy of the proposed method for $100$-dimensional analytic functions
}
\begin{center}
\begin{tabular}{|p{4.0cm}|p{3.2cm}|p{3.2cm}|p{3.2cm}|}
\hline
\multicolumn{1}{|c|}{Function}            &
\multicolumn{1}{ c|}{TT-rank}             &
\multicolumn{1}{ c|}{Error for $y_{min}$} &
\multicolumn{1}{ c|}{Time, sec.} \\ \hline


Exponential    &   1.0 & 0        &  35.5 \\
\hline

Grienwank      &   3.0 & 4.50e-14 &  41.0 \\
\hline

Qing           &   2.0 & 0        &  37.2 \\
\hline

Rastrigin      &   2.0 & 2.27e-13 &  36.7 \\
\hline

Schwefel       &   2.0 & 0        &  38.0 \\
\hline


\end{tabular}
\end{center}
\label{tab:result_function_big}
\end{table*}

Some of the benchmarks used in the previous experiment have a fairly simple analytic form (despite the very complex structure of the optima), and it is possible to explicitly construct the TT-cores for them.
Therefore, to analyze the performance of the algorithm in the essentially multidimensional case ($d = 100$), we perform the following steps for Exponential, Grienwank, Qing, Rastrigin, and Schwefel function:
\begin{enumerate}
    \item Set the Chebyshev grid with $N_1 = N_2 = \ldots = N_d = 2^{10}$ for function discretization (we select very fine grid to ensure that the minimum of the tensor is close to the real minimum of the function).
    \item Build the TT-tensor $\tens{Y} \in \set{R}^{N_1 \times N_2 \times \ldots \times N_d}$, which approximates the discretized function, using the corresponding explicit representation of its TT-cores and calculate the average TT-rank $r$ of $\tens{Y}$.
    \item Run \func{optima\_tt} algorithm for the TT-tensor $\tens{Y}$ and obtain the approximation to its minimum value ($y_{min}$).
    \item Find the expected minimum value of the TT-tensor $y_{min}^{(tens)} = \tens{Y}[\vect{i}_{min}^{(tens)}]$, where $\vect{i}_{min}^{(tens)}$ is the closest multi-index of the tensor to the known exact global minimum $\vect{x}_{min}^{(real)}$ of the function.
    \item Calculate the absolute error
    $
        e_{min} = |y_{min} - y_{min}^{(tens)}|
    $
    to check the accuracy.
\end{enumerate}
The computation results are presented in Table~\ref{tab:result_function_big}.
For all cases, the absolute error is not higher than $3 \cdot 10^{-13}$, and for three functions, we have obtained the exact value of the optimum, while the average time of one run was about $40$ sec.

\section{RELATED WORK}
\label{s:related}

Recently, TT-approach has been actively used in the direction of multivariable functions optimization, which is close to the TT-tensor optimization problem.
An iterative method based on the maximum volume approach is proposed in~\cite{sozykin2022ttopt, nikitin2022quantum}.
Submatrices of the maximum volume (i.\,e., submatrices having the maximum modulus of the determinant) are computed for successive unfoldings of the tensor, and then the search for the optimum is carried out among the elements of these submatrices.
The authors applied this approach to the problem of optimizing the weights of neural networks in the framework of reinforcement learning problems and to the quadratic unconstrained binary optimization problem.
Similar optimization approach was also considered in~\cite{selvanayagam2022global} with practical applications of the method for optimizing the housings of electronic devices, and in~\cite{shetty2022tensor} for optimizing the movement in space of robotic arms.
We also note work~\cite{soley2021iterative}, which considers an optimization method based on the iterative power algorithm in terms of the quantized version of the TT-decomposition (i.\,e., QTT-decomposition).

\section{DISCUSSION}
The theoretical evaluation of the performance of the \func{tt-optima} given in Theorem~\ref{trm:bad_case} states that the found approximate maximum can be exponentially smaller than the real maximum, and the parameter $K$ only additively reduces the effective dimensionality of the tensor in which the search is performed.
Note that the accuracy estimate given by Theorem~\ref{trm:bad_case} is strict: there are examples where our algorithm returns the value for which inequality~\eqref{eq:bad_case} turns into an equality.
However, in practice, as shown in our experiments, this does not happen and the algorithm works very efficiently, often finding the exact maximum in the problems under study.
In practice, such approaches are widespread in some problems and show their efficiency, such as beamsearch in NLP~\cite{zhu2022leveraging}, so the theoretical formulation of function properties under which the algorithm is efficient is an important future work. 
Thus, we hope that in practical applications our algorithm will be effective.

We can specify several heuristics as ways to increase accuracy.
First, it is possible to combine several first indices, thus reducing the dimensionality, as described in Corollary~\ref{col:join}.

Secondly, we can change the order in which the indices are passed.
The proof of the Theorem~\ref{trm:bad_case} heavily relies on the fact that the tensor dimensions are enumerated in a certain order, but the TT-structure allows us to consider
dimensions in any order, and even to choose the order at runtime.
Exploring possible modifications to the algorithm that would lead to better results in practice and possibly better theoretical estimates is also an important possible extension of the work.
In a practical implementation of the algorithm we make two passes, namely from one end of the TT-tensor and from the other end, and then choose the best one.

Third, it is possible to store as the TT-representation of the tensor~$\ty$ not the function which is to be maximized itself,
but some transformation of it.
In paper~\cite{soley2021iterative} it is proposed to
raise the function to the $p$-th power for some~$p$.
Other transformations can be proposed, that would ``pull out'' the maximum, e.\,g., taking an exponent.
In the course of the algorithm, we have already used natural squaring, as we worked with a tensor~$\tp$ instead of a tensor~$\ty$.
Thus, from the factor in inequality~\eqref{eq:bad_case} for the tensor~$\tp$ we actually have to extract the root to obtain the ratio for the maximum for the desired tensor~$\ty$.
As our future work, we point out the application of efficient algorithms for application of some monotonic transformations to the TT-tensor in conjunction with our algorithm.

In addition, note that increasing the parameter~$K$ in the algorithm, as experiments with random tensors show, does not lead to a significant increase in accuracy.
As can be seen from Fig.~\ref{fig:k_dep}, a value of $K=10$ is already sufficient for acceptable accuracy.

\section{CONCLUSION}
We presented a new method for optimizing TT-tensors with assessment of its complexity and convergence,
and we demonstrated its effectiveness for a number of model problems.
Our method allows a wide range of practical applications
in the field of global optimization
of multivariate functions and multidimensional data arrays.

The main advantage of our method is its high speed: it works in a single pass over the tensor train, performing several matrix-vector multiplications at each step, which can be paralleled.
We emphasize that the complexity of our method grows linearly with the dimensionality~$d$ of the input tensor, so we deal with the curse of dimensionality.



\section{ACKNOWLEDGMENT}
The work was supported by the Ministry of Science and Higher Education of the Russian Federation under grant No. 075-10-2021-068.

\section{APPENDIX}
\label{s:appendix}

\subsection{Basic properties of the TT-decomposition}

We formulate below the operations in the TT-format, which we used in the main text, in the form of corresponding lemmas.

\begin{lemma}\label{lem:tt_add}
Consider two tensors $\ty^{(1)} \in \set{R}^{N_1 \times \ldots \times N_d}$ and $\ty^{(2)} \in \set{R}^{N_1 \times \ldots \times N_d}$ of the same size, represented in the TT-format with TT-cores $\tg_i^{(1)} \in \set{R}^{R_{i-1}^{(1)} \times N_i \times R_i^{(1)}}$ and $\tg_i^{(2)} \in \set{R}^{R_{i-1}^{(2)} \times N_i \times R_i^{(2)}}$ ($i = 1, 2, \ldots, d$), respectively.
Then their element-wise sum $\ty = \ty_1 \oplus \ty_2 \in \set{R}^{N_1 \times \ldots \times N_d}$ can be represented in the TT-format with TT-cores $\tg_i \in \set{R}^{R_{i-1} \times N_i \times R_i}$ ($i = 1, 2, \ldots, d$):
\begin{equation*}
\begin{split}
\tg_1 [1, n_1, :] = & \begin{pmatrix}
    \tg_1^{(1)} [1, n_1, :] &
    \tg_1^{(2)} [1, n_1, :]
\end{pmatrix},
\\
\tg_i [:, n_i, :] = & \begin{pmatrix}
    \tg_i^{(1)} [:, n_i, :] & 0 \\
    0 & \tg_i^{(2)} [:, n_i, :]
\end{pmatrix},
\quad
\\
\tg_d [:, n_d, 1] = & \begin{pmatrix}
    \tg_d^{(1)} [:, n_d, 1] &
    \tg_d^{(2)} [:, n_d, 1]
\end{pmatrix},
\end{split}
\end{equation*}
for $i = 2, 3, \ldots, d-1$ and all $n_j = 1, 2, \ldots, N_j$ ($j = 1, 2, \ldots, d$).
\end{lemma}
We denote this operation as $\func{tt\_add}(\cdot, \cdot)$.
Note that the TT-ranks of the result are $R_0 = 1$, $R_i = R_i^{(1)} + R_i^{(2)}$ ($i = 1, 2, \ldots, d-1$), $R_d = 1$.
The element-wise difference $\ty = \ty_1 \ominus \ty_2 \in \set{R}^{N_1 \times \ldots \times N_d}$ may be computed in the similar way, if the last TT-core in $\ty^{(2)}$ is multiplied by $-1$, i.\,e., $\tg_d^{(2)} \gets (-1) \cdot \tg_d^{(2)}$. 
We denote this operation as $\func{tt\_dif}(\cdot, \cdot)$.


\begin{lemma}\label{lem:tt_const}
A tensor $\ty \in \set{R}^{N_1 \times \ldots \times N_d}$ whose elements are all identically equal to a given number $v$, can be represented in the TT-format with TT-cores of unit rank $\tg_i \in \set{R}^{1 \times N_i \times 1}$ ($i = 1, 2, \ldots, d$) equal to
\begin{equation*}
\begin{split}
\tg_i [1, :, 1] = & \begin{pmatrix}
    \sqrt[\leftroot{1}\uproot{2}d]{|v|} &
    \ldots &
    \sqrt[\leftroot{1}\uproot{2}d]{|v|}
\end{pmatrix},
\quad
\\
\tg_d [1, :, 1] = & \begin{pmatrix}
    \sqrt[\leftroot{1}\uproot{2}d]{|v|} &
    \ldots &
    \sqrt[\leftroot{1}\uproot{2}d]{|v|}
\end{pmatrix} \cdot \func{sign}(v),
\end{split}
\end{equation*}
for $i = 1, 2, \ldots, d-1$ (if $|v| = 0$, then all TT-cores should be identically zero tensors).
\end{lemma}
We denote such constant TT-tensor as $\func{tt\_const}\bigl((\indt N_d), \,v\bigr)$.

\begin{lemma}\label{lem:tt_orthogonalize}
Consider a tensor $\ty \in \set{R}^{N_1 \times \ldots \times N_d}$, represented in the TT-format with TT-cores $\tg_i \in \set{R}^{R_{i-1} \times N_i \times R_i}$ ($i = 1, 2, \ldots, d$).
If we apply the orthogonalization operation to its TT-cores according to Algorithm~\ref{alg:orthogonalize},
then the following relations will hold for the updated tensor TT-cores
\begin{equation}\label{eq:tt_orthogonalization}
\sum_{j=1}^{N_i}
    \;
    (\tg_i [:, j, :])
    \;
    (\tg_i [:, j, :])^T
    = \matr{I}_{R_{i-1}},
\end{equation}
for all $i = 2, 3, \ldots, d$, where $\matr{I}_{R_{i-1}}$ is a unit diagonal matrix of the size $R_{i-1} \times R_{i-1}$.
\end{lemma}
We denote this in-place operation as $\func{tt\_orth}(\cdot)$.

\begin{figure}[t!]
\begin{center}
\begin{algorithm}[H]
\small
\SetAlgoLined
\KwData{
    $d$-dimensional TT-tensor $\ty \in \set{R}^{N_1 \times N_2 \times \ldots \times N_d}$, presented as a list of $d$ TT-cores $\{ \tg_1,\,\tg_2,\,\ldots,\,\tg_d \}$ ($\tg_i \in \set{R}^{R_{i-1} \times N_i \times R_i}$ for $i = 1, 2, \ldots, d$).
}
\KwResult{
    TT-tensor $\ty$ with orthogonalized TT-cores.
}

\For{$i = d$ \KwTo $2$}{
    // Update the $i$-th TT-core:
    
    $
    \mg_i = \func{reshape} \left(
        \tg_i, \; (R_{i-1}, N_i \cdot R_i)
    \right)
    $
    
    $
    \matr{R}, \; \matr{Q} = \func{rq}(\mg_i)
    $
    // Compute the RQ-decomposition

    $
    \tg_i = \func{reshape} \left(
        \matr{Q}, \; (R_{i-1}, N_i, R_i)
    \right)
    $
    
    // Update the $(i-1)$-th TT-core:
    
    $
    \mg_{i-1} = \func{reshape} \left(
        \tg_{i-1}, \; (R_{i-2} \cdot N_{i-1}, R_{i-1})
    \right)
    $
    
    $
    \mg_{i-1} = \mg_{i-1} \matr{R}
    $
    
    $
    \tg_{i-1} = \func{reshape} \left(
        \mg_{i-1}, \; (R_{i-2}, N_{i-1}, R_{i-1})
    \right)
    $
}

\caption{Method \func{tt\_orth} for TT-tensor orthogonalization.}
\label{alg:orthogonalize}
\end{algorithm}
\end{center}
\end{figure}

\subsection{Proof of the Theorem~\ref{th:orth_property}}



Recall that we need to prove the following relationship:
\begin{multline}
\tp_l[\indt n_{d-l}] =
    \sumnNseqSmall n_d-l+1, d-l+2 .. d^N
        \\
        \ty[\indt n_d]\ty[\indt n_d].
\end{multline}
Substituting the representation of the tensor~$\ty$ 
in the TT-format~\eqref{eq:tt-repr-tns} into the last expression, we obtain
\begin{multline}
\tp_l[\indt n_{d-l}] =
\sumnNseqSmall n_d-l+1, d-l+2 .. d-1^N
\\
\tg_1 [1, n_1, \alld]
\;
\tg_2 [\alld, n_2, \alld]
\;
\cdots \times
\tg_{d-1} [\alld, n_{d-1}, \alld]
\\
\raisebox{0pt}[0pt][7ex]{$
\underbrace{
\sum_{n_d}
\tg_d [\alld, n_d, \alld]
\;
\tg_d [\alld, n_d, \alld]^T
}_{I_{R_{d-1}}}
$}
\;
\tg_{d-1} [\alld, n_{d-1}, \alld]^T
\\
\cdots
\tg_2 [\alld, n_2, \alld]^T
\;
\tg_1 [1, n_1, \alld].
\end{multline}
Given the orthogonality of the matrix $\tg_d [\alld, n_d, \alld]$ in the sense of~\eqref{eq:tt_orthogonalization}, we can remove them from the last relation, since they become a unit matrix
\begin{multline}
\tp_l[\indt n_{d-l}] =
\sumnNseqSmall n_d-l+1, d-l+2 .. d-1^N
\\
\tg_1 [1, n_1, \alld]
\;
\tg_2 [\alld, n_2, \alld]
\;
\cdots 
\tg_{d-1} [\alld, n_{d-1}, \alld]
\\
\tg_{d-1} [\alld, n_{d-1}, \alld]^T
\cdots
\tg_2 [\alld, n_2, \alld]^T
\\
\tg_1 [1, n_1, \alld].
\end{multline}
Similarly,
applying the orthogonality condition to matrices 
$\tg_{d-1} [\alld, n_{d-1}, \alld]$, $\tg_{d-2} [\alld, n_{d-2}, \alld]$, $\ldots$, $\tg_{d-l+1} [\alld, n_{d-l+1}, \alld]$, we finally obtain
\begin{multline}\label{eq:orth_convolv}
\tp_l[\indt n_{d-l}] =
\tg_1 [1, n_1, \alld]
\;
\tg_2 [\alld, n_2, \alld]
\\
\cdots
\tg_{d-l} [\alld, n_{d-l}, \alld]
\;
\tg_{d-l} [\alld, n_{d-l}, \alld]^T
\cdots
\\
\tg_2 [\alld, n_2, \alld]^T
\;
\tg_1 [1, n_1, \alld]
=
\\
\bigl\|
\tg_1 [1, n_1, \alld]
\;
\tg_2 [\alld, n_2, \alld]
\;
\cdots
\\
\tg_{d-l} [\alld, n_{d-l}, \alld] 
\bigr\|^2_2.
\end{multline}

Thus, the convenience of representing a TT-tensor in orthogonalized form in the context of our method is that we can simply branch off the TT-cores by the indices of which we convolve the tensor with itself.

\subsection{Proof of the Theorem~\ref{state:probabilistic_interpretation}}

Note that marginal distributions~\eqref{eq:marginal}
are expressed in terms of the overall distribution density function $\func{p}(\indt\xi_d)$ through the sums over the remaining arguments:
\begin{multline}
\func{p}_i(\xi_i \, | \, \indt\xi_{i - 1}) =
\sumnNseqSmall \xi_i+1, i+2 .. d^N
\\
\func{p}(\indt\xi_d),
\end{multline}
for $i=\rangeo d$.
We assume that the variables $\indt\xi_{i-1}$ are already found and fixed, and the variable~$\xi_i$ is to be found at the current step.

Let $C=\plusl\tp d$ be the sum of all elements of tensor $\tp$, i.\,e., normalization constant.
In the line~\ref{alg:line:Q_first} of Algorithm~\ref{alg:optima_tt_max} norms of the rows of the matrix $\mq$ are exactly the marginal distribution of the first coordinate
up to normalization constant
due to the orthoginality condition~\eqref{eq:orth_convolv}
\begin{equation}\label{eq:coonct_mat_and_p}
\normb{\mq[l,\alld]}_2^2 =
    \plusl\tp {(d-1)}[l]
    =
    C\func{p}_1(l).
\end{equation}

Next, if we selected some index $\tilde{\xi_1}$,
we then can directly obtain the TT-cores~$\{\tg'_i\}_{i=1}^{d-1}$ of the $(d-1)$-dimensional TT-tensor~$\tp^{(1)}$,
which corresponds to the conditional distribution $\func p'(\indft \xi_2,3 .. d \,|\, \tilde{\xi_1})$
\begin{multline}
\tp^{(1)}[\indft \xi_2,3..d]
=
\tp[\tilde{\xi_1},\,\indft \xi_2,3..d],\\
\tg'_1[1, n, \alld] = \tg_1[1, \tilde{\xi_1}, \alld] \tg_2[\alld, n, \alld]
\\
\textfor n=\rangeo N_1,
\\
\tg'_i = \tg_{i+1}, \quad \textfor i=\range 2,3,..d-1.
\end{multline}
In the line \ref{alg:line:Q_rest} of Algorithm~\ref{alg:optima_tt_max},
the matrix multiplication corresponding to the construction of the new first TT-core is performed.

Note that the obtained TT-decomposition is also orthogonal, since all but the first TT-cores coincide with those of the orthogonal TT-decomposition of the tensor $\tp$.
Thus, a relation similar to~\eqref{eq:coonct_mat_and_p} is also true for this decomposition, and we can sequentially apply the described procedure.
Namely, in the line~\ref{alg:line:Q_rest_reshape} of Algorithm~\ref{alg:optima_tt_max} squared norm of the rows of the  matrix $\mq$ corresponds to marginal distribution of the $i$-th variable~$\xi_i$, conditioned with the already chosen variables:
\begin{multline}
\normb{\mq[(j-1) K + l,\alld]}_2^2 = 
\\
\plusl\tp {(d-i)}
\bigl[
\range {I[j,\, 1]},I[j,\, 2]..  I[j,\, i - 1]
,\,l
\bigr]=\\
=C'\func p_i(l \,\vert \, 
\range {I[j,\, 1]},I[j,\, 2]..  I[j,\, i - 1]
).
\end{multline}
with corresponding normalization constant~$C'=
\plusl\tp {(d-i+1)}
\bigl[ I[j,\, 1], \,\ldots,\, I[j,\, i - 1] \bigr]
$.
We continue with these iterations until we reach the last core, after which we choose only the single best multi-index.
This completes the proof.

\begin{IEEEbiography}{Andrei Chertkov,}{\,}is a research engineer who develops efficient methods based on the tensor train decomposition for multidimensional data structures.
Contact him at a.chertkov@skoltech.ru.
\end{IEEEbiography}

\begin{IEEEbiography}{Gleb Ryzhakov,}{\,}is a research scientist with interests in computational tensor methods and low-rank tensor approximations.
Contact him at g.ryzhakov@skoltech.ru.
\end{IEEEbiography}

\begin{IEEEbiography}{Georgii Novikov,}{\,}is a research scientist with interests in the field of practical applications of low-rank tensor approximations in problems of machine learning and data analysis.
Contact him at georgii.novikov@skoltech.ru.
\end{IEEEbiography}

\begin{IEEEbiography}{Ivan Oseledets,}{\,}is a head of Center for Artificial Intelligence Technology. Professor Ivan Oseledets conducts research in the field of new methods of low-rank tensor approximations and tensor networks, effective forms of representation and methods of training deep neural networks, and development of applied AI solutions in the fields of computer vision, natural language processing, recommender systems, optimization, etc.
Contact him at i.oseledets@skoltech.ru.
\end{IEEEbiography}

\end{document}